\documentclass[reqno,11pt]{amsart}
\usepackage{color}

 \usepackage{mathtools}

\newtheorem{theorem}{Theorem}[section]

\newtheorem{corollary}[theorem]{Corollary}

\theoremstyle{definition}

\theoremstyle{remark}
\newtheorem{remark}[theorem]{Remark}
\newtheorem{example}[theorem]{Example}

 \makeatletter 
 \def\dashint{\operatorname{\,\,\,\mathclap{\int} \kern-.23em\text{\bf--}\!\!}}
 \makeatother

\def\dashnorm{\,\,\text{\bf--}\kern-.5em\|}
\def\ninf{\qopname\relax\@empty{inf\phantom{p}\!\!\!}}

 \makeatother

\newcommand\sfL{{\sf{L}}}

\newcommand\bB{\mathbb{B}}
\newcommand\bC{\mathbb{C}}

\newcommand\bM{\mathbb{M}}
\newcommand\bR{\mathbb{R}}

\newcommand\scB{\normalfont\textsc{b}}

\newcommand{\loc}{{\rm loc}\,}

 \newcommand{\mysection}[1]{\section{#1}
 \setcounter{equation}{0}}

\begin{document}

\title[On the heat equation with singular drift]
{On the heat equation with singular drift}
\author{N.V. Krylov}

\email{nkrylov@umn.edu}
\address{School of Mathematics, University of Minnesota, Minneapolis, MN, 55455}
 
 

\begin{abstract} 
We prove the maximum modulus estimates in terms of the $L_{q,p}$-norm
of the free term for solutions of the heat equation with Morrey drift for any $q,p$ satisfying $d/p+2/q<2$
and any order of integration in the definition
of the $L_{q,p}$-norm. An application to the case of $b$ satisfying the Ladyzhenskaya-Prodi-Serrin condition is given. The technique is
easily adaptable to equations with
Laplacians of order $\geq 1$.
 
\end{abstract}

\maketitle

\mysection{Introduction}

Let $\bR^{d}$, $d\geq2$, be a Euclidean space
of points $x=(x^{1},...,x^{d})$, $\bR^{d+1}
=\{(t,x):t\in\bR,x\in\bR^{d}\}$. Define
$$
D_{i}=\frac{\partial}{\partial x^{i}},\quad
Du=(D_{i}u).
$$

Fix $T\in(0,\infty)$ and set
$$
\bR^{d}_{T}=[0,T)\times \bR^{d}.
$$

We will be dealing with the maximum modulus estimates of solutions of
\begin{equation}
                      \label{3.27.1}
\partial_{t}u+\Delta u+b^{i}D_{i}u=-f
\quad \text{in} \quad \bR^{d}_{T},\quad
u(T,\cdot)=0.
\end{equation}
 All functions involved in this paper, such as $f,g,b,\scB$,
are assumed to be Borel measurable and {\em bounded\/}.
Therefore, \eqref{3.27.1} has a unique bounded
solution which belongs to $W^{1,2}_{p,\loc}
(\bR^{d}_{T})$ for any $p\in(1,\infty)$.
Our main interest is in the estimates like
\begin{equation}
                      \label{3.27.200} 
|u(0,0)|\leq N\|f\|_{L_{q,p}(\bR^{d}_{T})},
\end{equation}
where for $q,p\in(1,\infty)$
\begin{equation}
                      \label{4.3.1} 
\|f\|^{q}_{L_{q,p}(\bR^{d}_{T})}=
\int_{0}^{T}\Big(\int_{\bR^{d}}|f(t,x)|^{p}\,dx
\Big)^{q/p}\,dt.
\end{equation}
If $q$ or $p$ are infinite, one understands
this norm in a common way.

 The following is a kind of side-result almost explicitly stated in \cite{RZ_23} primarily devoted
to proving the existence and weak uniqueness of solutions of \eqref{3.28.1}.
\begin{theorem}
                        \label{theorem 4.1.1}
Let $d\geq 3$ and 
$$
\bar b:=\|b\|_{L_{q_{0},p_{0}}(\bR^{d}_{T})}<\infty
$$
 for some $p_{0},q_{0} $
satisfying (Ladyzhenskaya-Prodi-Serrin condition)
\begin{equation}
                            \label{4.3.5}
p_{0}\in(d,\infty],\quad\frac{d}{p_{0}}+\frac{2}{q_{0}}= 1.
\end{equation}
Then for each 
$p ,q  $
satisfying
\begin{equation}
                            \label{3.31.2}
p,q\in(1,\infty),\quad \frac{d}{p}+\frac{2}{q}<2
\end{equation}
estimate \eqref{3.27.200} holds
with $N=N'T^{1-d/(2p)-1/q}$ and $N'$ depending only on $d$, $\bar b,q,p$ if $p_{0}=
\infty$ and otherwise only on $d$, $\bar b,p_{0},q,p$ and the function
$$
b(t):=\Big(\int_{\bR^{d}}
|b (t,x)|^{p_{0}}\,dx\Big)^{1/(p_{0} -d )}.
$$ 
\end{theorem}

\begin{remark}
                       \label{remark 4.4.6}
The fact that \eqref{3.27.200} holds is almost
trivial (cf. Remark \ref{remark 4.3.3})
because our data are assumed to be bounded. What is highly
nontrivial is what the constant $N$ depends upon.
\end{remark}
\begin{remark}
                       \label{remark 4.3.3}
If $b\equiv0$ the result of Theorem
\ref{theorem 4.1.1} is achieved by trivial
computations.  It seems that the fact of existence of estimate \eqref{3.27.1}
for {\em all\/} $q,p$ satisfying \eqref{3.31.2},
given that $b\in L_{q_{0},p_{0}}$ with $q_{0},p_{0}$ satisfying \eqref{4.3.5}  was never addressed
before \cite{RZ_23}. As far as the author
is aware,  it was
not know that \eqref{3.27.200} is true even for
$q=q_{0},p=p_{0}$  if \eqref{4.3.5} holds. However, if it holds with $<$ in place of $=$,
it is relatively easy to prove that the assertion
of Theorem \ref{theorem 4.1.1} holds true.
This is done in Lemma 3.3 of \cite{KR_05}
by using the probability theory and Girsanov's transformation. It is worth pointing out that
the estimates like \eqref{3.27.200} are very widely discussed in probabilistic literature
related to investigating the weak and strong
solvability of \eqref{3.28.1}. Apart from \cite{RZ_23} we are attracting the reader's attention to
\cite{XXZZ_20} and numerous references in these papers, in particular, to \cite{BFGM_19}.
\end{remark}

The goal of this article is to generalize
Theorem \ref{theorem 4.1.1} by proving the following, (recall
that $d\geq2$) in which $\bB_{r}$ is the set
of balls $B$ in $\bR^{d}$ of radius $r$,
$|B|$ is the volume of $B$, and
$$
\dashnorm f\|^{p}_{L_{p}(B)}=\frac{1}{|B|}
\int_{B}|f|^{p}\,dx.
$$

 \begin{theorem}
                        \label{theorem 3.31.1}
Let $q,p$ satisfy \eqref{3.31.2} and suppose
 that for any   constant $\tilde b\in(0,1]$ 
there exists $p_{0}=p_{0}(q,p,\tilde b)\in (1,d]$, such that
$p_{0}':=p_{0}/(p_{0}-1)<(2q)\wedge(2p)$, and
the function
$b$ admits the representation 
$$
b=b'+\scB,\quad (b',\scB)=(b',\scB)_{b,q,p,\tilde b}
$$
 such that
$$
\sup_{t\in\bR}\sup_{r>0}r
\sup_{B\in \bB_{r}} 
\dashnorm b'(t,\cdot) \|_{L_{p_{0}}(B)}\leq\tilde b,\quad [\scB]^{2}:=\int_{0}^{T}
\sup_{\bR^{d}}|\scB(t,x)|^{2}\,dt<\infty.
$$
Then  
the solution $u$ of \eqref{3.27.1}
admits the estimate  \begin{equation}
                      \label{3.27.2}
\sup_{\bR^{d}_{T} }|u |\leq N_{0}T^{1-d/(2p)-1/q}
 \|f\|_{L_{q,p}(\bR^{d}_{T})} ,
\end{equation}
 where $N_{0}$
depends only on $d,p,q$ and the functions
$p_{0}(q,p,\tilde b), [\scB_{b,q,p,\tilde b}]$.
\end{theorem}

This theorem is proved in Section \ref{section 4.3.1}.
Let us show, following Remark 2.2 in \cite{Kr_25_2}
or Remark 2.2  in \cite{Kr_25_1}, how it implies Theorem \ref{theorem 4.1.1}.

{\bf Proof of Theorem \ref{theorem 4.1.1}
for $d\geq2$}.
Let $q,p$ satisfy \eqref{3.31.2}.
In case $p =\infty$  it suffices to take $b'=0,\scB=b$ in Theorem \ref{theorem 3.31.1}.
Since otherwise $p \in(d,\infty) $, we   take an
arbitrary constant $\hat N$, introduce
$
\lambda(t)=\hat Nb(t),
$ 
and then for  
$$
b'  (t, x):=b  (t, x)I_{|b  (t,x)|\geq \lambda(t)}
$$
and $B\in\bB_{\rho}$  we have
$$
\dashint_{B }|b' (t, x)|^{d}\,dx
\leq \lambda^{d-p }(t)
\dashint_{B }|b  (t, x)|^{p }\,dx
\leq N(d)\hat N^{d-p_{0}}\rho^{-d}.
$$
 
Here $N(d)\hat N^{d-p  }$ can be made
arbitrarily small if we choose $\hat N$
large enough. In addition, for 
$\scB: =b -b' $ we have
 $|\scB|\leq \lambda $ and
$$
\int_{0}^{T}\lambda^{2}(t)\,dt
=\hat N^{2}\int_{0}^{T}\Big(\int_{\bR^{d}}|b (t,x)|^{p  }
\,dx\Big)^{q  /p  }\,dt<\infty.
$$
It follows that the conclusion of Theorem
\ref{theorem 3.31.1} is applicable
as long as $2q,2p>d'$. Observe that
\eqref{3.31.2} implies that $2q,2p>2$
and $d\geq 2$, so that $d'\leq 2$.   \qed
\begin{remark}
                        \label{remark 4.1.1}
It follows from the above proof of Theorem 
\ref{theorem 4.1.1}  that its conclusions are also
valid if $b$ is represented as the finite sum of terms
each of which satisfies the assumption of  
Theorem \ref{theorem 4.1.1} with perhaps
$q_{0},p_{0}$ different for different terms.
\end{remark}

Here is an example showing the case in which
$|b(t,x)|\leq b(x)$, $b\not\in L_{d-1/2,\loc}(\bR^{d})$ and Theorem \ref{theorem 3.31.1} is applicable. In this case  
 the assumption of Theorem \ref{theorem 4.1.1}
is not satisfied by far.
The following is close to Example 2.3 in \cite{Kr_25_2} or
 Example 2.3   in \cite{Kr_25_1}. Below
 $$
 B_{r}=\{x\in\bR^{d}:|x|<r\},\quad
 B_{r}(x)
=x+B_{r}.
$$

\begin{example}
                        \label{example 9.25.1}

Let $d\geq3$, $p_{0} \in[d-1,d)$   and take $r_{n}>0$, $n=1,2,...$, such that
the  sum of $\rho_{n}:=r_{n}^{d-p_{0}}$ is $1/2$. 
Also take a sequence $\alpha_{n}\downarrow 0$ such that
$$
\sum_{n=1}^{\infty} r_{n} ^{d-p_{0}}\alpha_{n}^{p_{0}}<\infty,
\quad  \sum_{n=1}^{\infty}  r_{n} ^{d-p}\alpha_{n}^{p}=\infty
$$
for any $p>p_{0}$. For instance, for large $n$
set $r_{n}=(n\ln^{3}n)^{- 1/(d-p_{0})},\alpha_{n}=(\ln n)^{-1/p_{0}}$.

Next,
let $e_{1}$ be the first
basis vector  and set   $x_{0}=1$, $\rho_{n}=r_{n}^{d-p_{0}}$,
$$
x_{n}=1-  2\sum_{1}^{n}\rho_{i} ,\quad 
c_{n}=(1/2)(x_{n}+x_{n-1}),
$$
$$
  b_{n}(x)=( \alpha_{n}/r_{n}) I_{B_{r_{n}}(c_{n}e_{1})},\quad b=\sum _{1}^{\infty}b_{n}.
$$
Since $r_{n}\leq 1$ and $d-p_{0}\leq 1$,  the supports of $b_{n}$'s are disjoint and
for $p>0$
$$
\int_{B_{1}}b^{p}\,dx=\sum _{1}^{\infty} 
N(d)( \alpha_{n}/r_{n} )^{ p}r_{n}^{d},
$$
which is finite for $p=p_{0}$ and infinite for $p>p_{0}$.

Then observe that for any $n\geq 1$ and any ball $B$
of radius $\rho$
\begin{equation}
                                \label{4.1.1}
 \int_{B }  b_{n} ^{p_{0}}dx \leq N(d)\alpha_{n}^{p_{0}} \rho^{d-p_{0}} .
\end{equation}
Also note that the volume of the intersection of $B$ with $B_{r_{n}}(c_{n}e_{1})$ will increase   if we move $B$ perpendicularly
to $e_{1}$ so that its new center will be
on the $x^{1}$-axis. Therefore, while
estimating the integral of $b^{p}$ over $B$,
we may assume that the center of $B$ is
on the $x^{1}$-axis. 
Then for any integer $k\geq 2$, if the intersection of $B$ with $\Gamma_{k}:=\bigcup_{n\geq k} B_{r_{n}}(c_{n}e_{1})$
is nonempty, the intersection
 consists of some nonempty $B\cap B_{r_{i}}(c_{i}e_{1})\subset B$, $i=i_{0},...,i_{1}\geq k$, and, possibly, $B\cap B_{r_{i_{0}-1}}(c_{i_{0}-1}e_{1})$ and
$B\cap B_{r_{1}+1}(c_{r_{i}+1}e_{1})$.
According to \eqref{4.1.1}
$$
\int_{B}[b^{p_{0}}_{r_{i_{0}-1}}+
b^{p_{0}}_{r_{i_{1}+1}}]\,dx\leq N(d)\alpha^{p_{0}}_{k-1}
\rho^{d-p_{0}}.
$$
Regarding other $B_{r_{i}}(c_{i}e_{1})$ note that, obviously,
$$
   (2\rho)\wedge 1\geq c_{i_{0}}-c_{i_{1}}=2 \sum_{i=i_{0}+1}^{i_{1} }
\rho_{i}+\rho_{i_{0}}-\rho_{i_{1}}\geq \sum_{i=i_{0} }^{i_{1} }
\rho_{i}.
$$
Therefore,
$$
 \int_{B} ( bI_{\Gamma_{k}})  ^{p_{0}}\,dx\leq 
 N(d)\sum_{i=i_{0}}^{i_{1}}(\alpha_{i}/r_{i})^{p_{0}}r_{i}^{d} +
N(d)\alpha^{p_{0}}_{k-1}
\rho^{d-p_{0}}
$$
$$
\leq N(d)\alpha^{p_{0}}_{k-1}\big((2\rho)\wedge 1+\rho^{d-p_{0}}\big)) 
$$
where the last term is less than $N(d)\alpha^{p_{0}}_{k-1}\rho^{d-p_{0}}$. 

We see that $b=bI_{\Gamma_{k}}+bI_{\Gamma^{c}_{k}}$, where
$$
 \sup_{\rho>0}\rho
\sup_{B\in \bB_{\rho}} 
\dashnorm bI_{\Gamma_{k}} \|_{L_{p_{0}}(B)}
$$
can be made as small as we like and 
$bI_{\Gamma^{c}_{k}}$ is bounded.
Moreover, for $q,p$ satisfying \eqref{3.31.2},
we have $2p,2q>2$ and $p_{0}'\leq (d-1)/(d-2)
\leq2$.
Therefore, Theorem \ref{theorem 3.31.1}
is applicable in case $b(t,x)$ is such that
$|b(t,x)|\leq b(x)$.
\end{example}

\mysection{Auxiliary results}

Define 
$$
C_{r}(t,x)=[t,t+r^{2})\times B_{r}(x),
$$
and let $\bC_{r}$ be the collection of $C_{r}(t,x)$. If $C\in\bC_{r}$ by $|C|$ we mean
its Lebesgue measure and for appropriate $f$
and $p\geq1$ we set
$$
\dashint_{C}f\,dxdt=\frac{1}{|C|}
\int_{C}f\,dxdt,\quad \dashnorm f\|^{p}_{L_{p}(C)}=\dashint_{C}|f|^{p}\,dxdt.
$$

For $k,s,r,\alpha>0$, and appropriate $f(t,x)$'s
on $\bR^{d+1}$
\index{$S$@Miscelenea!$p_{\alpha,k}(s,r)$}%
\index{$C$@Operators!$P_{\alpha,k}f(t,x)$}%
 define
$$
p_{\alpha,k}(s,r)=\frac{1}{s^{(d+2-\alpha)/2}}e^{-r^{2}/(ks)}I_{s>0}, 
$$
$$
P_{\alpha,k}f(t,x)=\int_{\bR^{d+1} }p_{\alpha,k}(s,|y|)f(t+s,x+y)\,dyds.
$$
$$
=\int_{t}^{\infty}\int_{\bR^{d} }p_{\alpha,k}(s-t,|y-x|)
f(s,y)\,dsdy.
$$  

The following is Theorem 3.1 of \cite{3}.

\begin{theorem}
                     \label{theorem 10.5,1}
(i) There is a constant $c(d)>0$ such that
$u=c(d)P_{2,4}(\partial_{t}u+\Delta u)$
if $u\in C^{\infty}_{0}(\bR^{d+1})$.

(ii) For $\alpha,\beta,k>0$ we have
$P_{\alpha,k}P_{\beta,k}=c(\alpha,\beta,k)P_{\alpha+\beta,k}$. 

(iii) For any integer $n\geq1$, $\alpha>n$, and bounded $f$
with compact support we have $|D^{n}P_{\alpha,k}f|\leq N(d,\alpha,n)P_{\alpha-n,2\kappa}|f|$.

\end{theorem}

For $p_{0}\in(1,d]$, $\alpha\in(0,d/p_{0}]$, and a 
real-valued or $\bR^{d}$-valued $b$ given on $\bR^{d+1}$
define ($(\alpha,p_{0})$-Morrey norm)
$$
\tilde b_{\alpha,p_{0}}:=\sup_{t\in\bR}\sup_{r>0}r^{\alpha}
\sup_{B\in \bB_{r}} 
\dashnorm b(t,\cdot) \|_{L_{p_{0}}(B)}.
$$
 
For $r,s\in(1,\infty)$ we define the space
$L_{r,s}$ as the set of function $f$ on $\bR^{d+1}$ with 
$\|f\|_{L_{r,s}}<\infty$, where the norm is
defined according to
\eqref{4.3.1} with $r,s$ in place of $q,p$
and the first integral   taken over $\bR$
instead of $(0,T)$.

\begin{theorem} 
                        \label{theorem 3.27,1}
Let  $b(t,x)\geq0$ and $r,s >p_{0}':=p_{0}/(p_{0}-1)$. Then for any $f(t,x)\geq0$
\begin{equation}
                            \label{5.25,1}
\| P_{ \alpha,k}(bf)\|_{L_{r,s }}
\leq N\tilde b_{\alpha,p_{0}}
\|f\|_{L_{r,s}},
\end{equation}
where $N$ depends only on $d,\alpha,r ,s,p_{0},k$.
 
\end{theorem}

Proof. We may assume that $b,f$ are bounded
and have compact support. Then Minkowski's
inequality and the fact that, for any cylinder
$C$ we have
$P_{ \alpha,k}I_{C}(0,0)<\infty$, show  that 
$P_{ \alpha,k}(bf)\in L_{r,s}$. Then by
the Dong-Kim Theorem 6.2 of \cite{Kr_26_1}
$$
\| P_{ \alpha,k}(bf)\|_{L_{r,s }}\leq N
\| P^{\sharp}_{\alpha,k}(bf)\|_{L_{r,s }}.
$$ 
By Theorem 4.6 of \cite{Kr_26_1} the last norm
is dominated by a constant times the $L_{r,s }$-norm of
$$
\bM_{\alpha}(bf)(t,x):=\sup_{r>0} r^{\alpha}
\sup_{\substack{C\in \bC_{r}\\
C\ni (t,x)}}\dashint_{C }bf\,dyds.
$$
Here,  if $C=C_{r}(t,x)$, in the last integral  
$$
\dashint_{B_{r}(x)}(bf)(s,y)\,dy
\leq \Big(\dashint_{B_{r}(x)}b^{p_{0}}(s,y)\,dy\Big)^{1/p_{0}}
\Big(\dashint_{B_{r}(x)}f^{p'_{0}}(s,y)\,dy\Big)^{1/p'_{0}}
$$
implying that
$$
r^{\alpha} \dashint_{C_{r}(t,x)}bf\,dyds
\leq \tilde b_{\alpha,p_{0}} \Big(\dashint_{C_{r}(t,x)} f^{p'_{0}}(s,y)\,dyds\Big)^{1/p'_{0}} .
$$
Hence
$$
\| P_{ \alpha,k}(bf)\|_{L_{r,s }}\leq N
\tilde b_{\alpha,p_{0}}\| \bM(f^{p_{0}'})\|^{1/p'_{0}}_{L_{r/p_{0}',s/p_{0}' }},
$$
where $\bM$ is the maximal Hardy-Littlewood
operator. Since $r,s >p_{0}'$, by Theorem
6.1 of \cite{Kr_26_1} the last norm is dominated by a constant times
$$
\|   f^{p_{0}'} \|^{1/p'_{0}}_{L_{r/p_{0}',s/p_{0}' }}=\| f\|_{L_{r,s }}.
$$
This proves the theorem. \qed

\begin{remark}
                      \label{remark 3.30.1}
(i) Cleary, if $p_{0}'\leq
 2$, $r,s>1$ and $d/s+2/r< 1$, then 
$r,s>p'_{0}$.

(ii) The norm of the operator $P_{\alpha,k}:
L_{r,s}\to L_{r,s}$ is easily shown to go to
infinity as $r\to\infty$ and $s=$const or as $s\to\infty$ and $r=$const. Therefore, the above
computations show that the same holds for
the operator $\bM$. However, the norm $N_{cr,cs}$
of $\bM:L_{cr,qs}\to L_{cr,cs}$ tends to one
as $c\to \infty$. This follows from H\"older's
inequality implying that $N_{cr,cs}\leq
N_{r,s}^{1/c}$ for $c\geq1$.

\end{remark}

\begin{corollary}
                 \label{corollary 10.5,1}
Estimate \eqref{5.25,1} says that the operator
$f\to P_{\alpha,k}(bf)$ is bounded in $L_{r,s}$. Its conjugate (with time reversed)
is then also bounded as an operator in 
$L_{r',s' }$,   that is, if $1<r',s'<p_{0}$, then
$$
\|bP_{\alpha,k}f\|_{L_{r',s'}}
\leq N\|b\|_{\dot E_{p_{0},\alpha} }
\|f\|_{L_{r',s'}}.
$$
 
\end{corollary}

\begin{remark}
                     \label{remark 4.5.3}
Literally repeating the proof of Theorem \ref{theorem 3.27,1}, one shows that  under its
assumption  for any $f(t,x)\geq0$
\begin{equation}
                            \label{5.25,1}
\| P_{ \alpha,k}(bf)\|_{\sfL_{s,r }}
\leq N\tilde b_{\alpha,p_{0}}
\|f\|_{\sfL_{s,r }},
\end{equation}
where $N$ depends only on $d,\alpha,r ,s,p_{0},k$ and 
$$
\|f\|^{s}_{\sfL_{s,r }}=\int_{\bR^{d}}
\Big(\int_{\bR}|f(t,x)|^{r}\,dt\Big)^{s/r}\,dx.
$$
\end{remark}

\mysection{Proof of Theorem \ref{theorem 3.31.1}}
                    \label{section 4.3.1}

\begin{theorem}
                     \label{theorem 3.27,2}
Assume that  
\begin{equation}
                      \label{3.27.3}
p,q\in(1,\infty),\quad \frac{d}{p}+\frac{2}{q}<2,\quad 2p,2q>p_{0}'.
\end{equation}
 Then there is a constant $\tilde b>0$
depending only on $d,q,p,p_{0}$ such that, if 
$\tilde b_{p_{0}}:=\tilde b_{1,p_{0}}\leq\tilde b$, then for 
  any bounded $f\geq 0$ with compact support,
vanishing for $t> T$,
and the solution $u$ of
\eqref{3.27.1} estimate \eqref{3.27.2}
holds  with
 $N_{0}$ depending only on $d,q,p,p_{0}$.

\end{theorem}

Proof. First let $T=1$, but keep notation $T$ (for $1$)
 to show its role better.
Observe that
$$
\frac{d}{p}+\frac{2}{q}-1<\frac{d}{2p}+\frac{2}{2q}<1.
$$

Then we
know that $v:=P_{2,4}f$ is in $ W^{1,2}_{q,p}(  \bR^{d}_{T})$, satisfies $\partial_{t}v+\Delta v
+f=0$, vanishes for $t>T$, and
$$
\|\partial v,D^{2}v,Dv,v\|_{L_{q,p}(  \bR^{d}_{T})}\leq 
N\|f\|_{L_{q,p}(  \bR^{d}_{T})}.
$$
By embedding theorems (see Theorem 10.2
of \cite{BIN_75}) 
$$
\|DP_{2,4}f\|_{L_{2q,2p}(\bR^{d}_{T})}\leq N\|f\|_{L_{q,p}(\bR^{d}_{T})},
\quad |P_{2,4}f(0,0)|\leq N\|f\|_{L_{q,p}(  \bR^{d}_{T})}.
$$

Then observe that
$$
u=P_{2,4}(b^{i}D_{i}u)-P_{2,4}f,\quad
|Du|\leq NP_{1,8}(|b|\,|Du|)+|DP_{2,4}f|.
$$
It follows from Theorem \ref{theorem 3.27,1} that
$$
\|Du\|_{L_{2q,2p}(\bR^{d}_{T})}\leq
N(d,q,p,p_{0})\tilde b_{p_{0}}\|Du\|_{L_{2q,2p}( \bR^{d}_{T})}+N\|f\|_{L_{q,p}(  \bR^{d}_{T})}.
$$
With $N(d,q,p,p_{0})\tilde b_{p_{0}}\leq 1/2$ we get
$$
\|Du\|_{L_{2q,2p}(\bR^{d}_{T})}
\leq N\|f\|_{L_{q,p}(\bR^{d}_{T})},\quad 
\|P_{1,8}(|b|\,|Du|)\|_{L_{2q,2p}(\bR^{d}_{T})}\leq N\tilde b_{p_{0}}
\|f\|_{L_{q,p}(\bR^{d}_{T})}.
$$

After that set $g=P_{1,8}(|b|\,|Du|)$ and note that 
$$
|P_{2,4}(b^{i}D_{i}u)(0,0)|
\leq NP_{1,8}g(0,0)\leq \|g\|_{L_{2q,2p}(\bR^{d}_{T})}\|p_{1,8}\|_{L_{(2q)',(2p)'}(\bR^{d}_{T})}.
$$
 An elementary
computation shows that the last norm is finite
and hence
$$
|u(0,0)|\leq |P_{2,4}(b^{i}D_{i}u)(0,0)|+|P_{2,4}f(0,0)|\leq N\|f\|_{L_{q,p}}.
$$
Similarly one estimates $|u|$ at any other
point in $\bR^{d}_{T}$. This proves the theorem
for $T=1$. For other values of $T$
the result is obtained by using self-similar transformations.
\qed

\begin{theorem}
                        \label{theorem 3.27,3}
Suppose that for a function $b$, 
$q,p\in(1,\infty)$,  for a  bounded $f\geq0$ with compact support
vanishing for $t>T$
and the solution of \eqref{3.27.1}
we have \eqref{3.27.2} 
with some constant $N_{0}$. Let
$\scB=\scB(t,x)$
be an $\bR^{d}$-valued function on
$\bR^{d+1}$. Then for   the solution $v$ of
\begin{equation}
                      \label{3.27.10}
\partial_{t}v+\Delta v+(b^{i}+\scB^{i})D_{i} 
v=f
\quad \text{in} \quad \bR^{d}_{T},\quad v(T,\cdot)=0
\end{equation}
we have
\begin{equation}
                      \label{3.27.20}
\sup_{  \bR^{d}_{T}}|v |\leq \sqrt2 e^{[\scB]^{2}/2}N_{0}T^{1-d/(2p)-1/q}
\|f\|_{L_{q,p}},
\end{equation}
where  
$$
[\scB]^{2}:=\int_{0}^{T}\sup_{\bR^{d}}|\scB(t,x)|^{2}\,dt .
$$
\end{theorem}

Proof. To make the argument simpler we assume
that $T=1$ claiming that the case of general $T$ is taken care of by self-similar
 transformations. Still
we use $T$ (for $1$) to show its role better.
Let $\Omega$ be the set of functions
$(t+\cdot,x_{\cdot})=\{(t+s,x_{s}),s\geq0\}$, where $t\in\bR$ and $x_{\cdot}$ is an $\bR^{d}$-valued continuous
function. It is a Polish space with metric
$$
\rho\big((t'+\cdot,x'_{\cdot}),(t''+\cdot,x''_{\cdot})\big)=|t'-t''|+\sup_{s\geq0}
|x'_{s}-x''_{s}|.
$$
It is well known that, due to the boundedness
of $b$, there exist probability measures $P_{t,x}$, $t\in\bR,x\in\bR^{d}$, on $\Omega$
such that for each $t\in\bR,x\in\bR^{d}$ with $P_{t,x}$-probability  one
the  equation
\begin{equation}
                             \label{3.28.1}
x_{s}=x+\sqrt 2 w_{s}+\int_{0}^{s}b(t+s,x_{s})\,ds,
\end{equation}
holds, where $w_{s}$ is a $d$-dimensional Wiener process relative to $P_{t,x}$. This fact is easily
obtained, for instance, by taking any Wiener
process and making   appropriate changes
of measure based on Girsanov's theorem.
Then one also obtains that solutions of 
\eqref{3.28.1} form a Markov process. 
Girsanov's theorem implies that for any nonnegative $f(t,x)$ and $(t,x)\in\bR^{d+1}$
we have
\begin{equation}
                             \label{3.28.2}
E_{t,x}\int_{0}^{T}f(t+s,x_{s})\,ds
=Ee^{\phi_{T}(t,x)}\int_{0}^{T}f(t+s,x+\sqrt2w_{s})\,ds,
\end{equation}
where $w_{\cdot}$ is a $d$-dimensional Wiener
process  and
$$
\phi_{T}(t,x)=2^{-1/2}\int_{0}^{T}
b(t+s,x+\sqrt2w_{s})\,dw_{s}-(1/4)\int_{0}^{T}
|b(t+s,x+\sqrt2w_{s})|^{2}\,ds.
$$
It is also well known that for any
$\lambda\geq0$ and $t\leq T$
\begin{equation}
                             \label{3.28.3}
E e^{\lambda\phi_{T-t}(t,x)}\leq e^{\lambda^{2}[\scB]^{2}/4}.
\end{equation}
This and the H\"older's inequality
show (see more details later) that the left-hand side of \eqref{3.28.2} admits the same $L_{q,p}$-estimates as if we had $b\equiv0$.
This allows us to use It\^o's formula
and for the solution $u$ of \eqref{3.27.1}
and $(t,x)\in\bR^{d}_{T}$
obtain that
$$
u(t,x)=E_{t,x}\int_{0}^{T-t}f(t+s,x_{s})\,ds.
$$
Now the assumed estimate \eqref{3.27.2}
and the Markov property of $(t+\cdot,x_{\cdot})$
imply that (the standard argument) for $f\geq0$
$$
E_{0,0}\Big(\int_{0}^{T}f(s,x_{s})\,ds\Big)^{2}
=2E_{0,0}\int_{0}^{T}f(s,x_{s})u(s,x_{s})\,ds
$$
$$
\leq 2 \sup_{\bR^{d}_{T}}|u |^{2}\leq 2N^{2}_{0}\|f\|^{2}_{L_{q,p}}.
$$

Finally, by Girsanov's theorem
$$
v(t,x)=E_{t,x}e^{\psi_{T-t}(t )}
\int_{0}^{T-t}f(t+s,x_{s})\,ds,
$$
where
$$
\psi_{T-t}(t )=
2^{-1/2}\int_{0}^{T-t}
\scB(t+s,x_{s})\,dw_{s}-(1/4)\int_{0}^{T-t}
|\scB(t+s,x_{s})|^{2}\,ds.
$$
It follows that
$$
|v(0,0)|\leq \Big(E_{0,0}e^{2\psi_{T-t}(t )}
\Big)^{1/2}\Big(E_{0,0} \Big(
\int_{0}^{T-t}f(t+s,x_{s})\,ds\Big)^{2}\Big)^{1/2}
$$
$$
\leq \sqrt2 e^{[\scB]^{2}/2}N_{0}\|f\|^{2}_{L_{q,p}}.
$$
Similarly one estimates $|v|$ at any other
point in $\bR^{d}_{T}$. \qed

{\bf End of proof of Theorem \ref{theorem 3.31.1}}. By using the maximum
principle and the fact that $f\leq|f|$ we conclude that it suffices to prove the 
theorem for $f\geq0$.
Take $q,p$ satisfying \eqref{3.27.3}
and take the corresponding constant $\tilde b$
from Theorem \ref{theorem 3.27,2}
and the function  $b'$ from the assumptions of Theorem \ref{theorem 3.31.1}.
Then by Theorem \ref{theorem 3.27,2}  
the solution $u$ of \eqref{3.27.1} with $b'
 $ in place of $b$ admits estimate \eqref{3.27.2} and by Theorem \ref{theorem 3.27,3}
the solution $v$ of \eqref{3.27.10} admits
estimate \eqref{3.27.20}. This proves the theorem. \qed

\begin{remark}
                 \label{remark 4.5.7}
Remark \ref{remark 4.5.3} allows us to
make the same arguments as above replacing
$L_{q,p}$ with $\sfL_{p,q}$ and prove
that in the assertion of Theorem \ref{theorem 4.1.1} regarding \eqref{3.27.200} one can replace $L_{q,p}$ with $\sfL_{p,q}$. Sometimes
it might be important. For instance, for $d\geq 3$
 define
$$
f(t,x)=I_{|x|\leq 1,0<t<1} 
|\,|x|-t|^{-2/d}.
$$
Then for any $q,p$ satisfying \eqref{3.31.2}
we have $p>d/2$, $2p/d>1$ and for any $t\in(0,1)$
$$
\int_{|x|<1}f^{p}(t,x)\,dx=N\int_{0}^{1}
r^{d-1}|r-t|^{-2p/d}\,dr=\infty.
$$
 Therefore, $f\not\in L_{q,p,\loc}$
and one cannot tell by using Theorem
\ref{theorem 4.1.1} that equation
\eqref{3.27.1} with $f\wedge n$ will
have solutions bounded by a constant
depending only on the data as in Theorem
\ref{theorem 4.1.1}.

However, for $1<q<d/2$ 
and any $x\in B_{1}$
$$
\int_{0}^{1}|\,|x|-t|^{-2q/d}\,dt\leq
\int_{-1}^{1}|t|^{-2q/d}\,dt<\infty,
$$
so that $f\in \sfL_{p,q}$. Therefore,
equation
\eqref{3.27.1} with $f\wedge n$  does
have solutions bounded by a constant
depending only on the data as in Theorem
\ref{theorem 4.1.1}.

By the way, for $d=2$ equation \eqref{3.31.2}
imposes the same restriction on $q,p$ and such  effect  is, obviously, impossible.
\end{remark}

\mysection{One more example}

In \cite{RZ_23} there are much more results
and we want to discuss one more of them.
The authors prove that if $d\geq3$ and
\begin{equation}
                               \label{4.4.2}
\sup_{t\in[0,T]}\sup_{\lambda>0, B\in\bB_{1}}
|\{x\in B:|b(t,x)|>\lambda\}|^{1/2}
\end{equation}
is small enough, then  the assertion of
Theorem \ref{theorem 4.1.1} holds true.
In the following example we show that
this does not hold if $d=2$.

\begin{example}
                      \label{example 3.28.1}
For $d\geq2$, $r,x\geq0,t>0$, $\theta\in(0,1)$ set
$$
\alpha=\theta(d-1),\quad
c^{-1}=\int_{\bR}e^{-y^{2}}\int_{0}^{\infty}
r^{\alpha-1}e^{-r^{2}}\,dr,
$$
\begin{equation}
                  \label{8,20.10}
p(t,x,r)=ct^{-(\alpha+1)/2}r^{\alpha}
\int_{-\pi/2}^{\pi/2}\big(\cos^{\alpha-1}
\phi\big) \exp\Big(-\frac{x^{2}+r^{2}-2xr
\sin\phi}{t}\Big)\,d\phi
\end{equation}
and for smooth $f(t,r)\geq0$ given on $[0,T]\times[0,\infty)$ 
introduce
$$
u(t,x)=\int_{0}^{T-t}\int_{0}^{\infty}p(s,|x|,r)f(t+s,r)
\,drds.
$$
The computations in Section 5.2 of \cite{Kr_25}
  show that in $\bR^{d}_{T}$ 
\begin{equation}
                    \label{3.29.1}
4\partial_{t}u+\Delta u+b^{i}D_{i}u+g=0,
\end{equation}
where $g(t,x)=f(t,|x|)$ and $b^{i}=-(d-1)(1-\theta )x^{i}/|x|^{2}$. Observe that for such
$b$ the quantity \eqref{4.4.2} is finite
and is as small as we wish if $\theta$ is close
to one (and $\alpha$ is close to $d-1$).

While estimating $u(0,0)$  observe that $(p'=p/(p-1))$
$$
\int_{0}^{\infty}p(s,0,r)f(s,r)
\,dr
$$
$$
\leq\Big(\int_{0}^{\infty}f^{p}(s,r)r^{d-1}\,dr\Big)^{1/p}\Big(\int_{0}^{\infty}p^{p'}(s,0,r) r^{-(d-1)/(p-1)}\,dr\Big)^{1/p'}.
$$
Here the first factor equals a constant times
$\|g(s,\cdot)\|_{L_{p}}$. The $p'$-th power
of the second one equals a constant times
$$
s^{-p'(\alpha+1)/2}\int_{0}^{\infty}r^{ \alpha p'-(d-1)/(p-1)}
e^{-r^{2}p'/s} \,dr,
$$
which is infinite for $p\leq d/(\alpha+1)$
and otherwise equals another constant
times
$$
s^{-p'(\alpha+1)/2+[\alpha p'-(d-1)/(p-1)+1]/2}.
$$
It is seen that, if $d=2$, no matter how
small $1-\theta =1-\alpha >0$ is,
estimate \eqref{3.27.200} fails to hold for
 $p=2/(1+\alpha)>1$ and an appropriate $q$
such that $2/p+2/q<2$.

\end{example}

{\bf Declarations}. 
No funds, grants, or other support was received.
The author has no relevant financial or non-financial interests to disclose.
The manuscript contains no data.

\end{document}